\documentclass{article}
\usepackage{amsfonts,amsmath,amssymb,amsthm}

\usepackage{hyperref}

\usepackage{graphics, epsfig}

\usepackage{titling}
\predate{}
\postdate{}

\title{The $4 \times 4$ minors of a $5 \times 5$ symmetric matrix are a tropical basis}
\author{Dylan Zwick \\ \href{mailto:zwick@math.utah.edu}{zwick@math.utah.edu}}
\date{}

\begin{document}

\maketitle

\begin{abstract} 
  This paper proves the $4 \times 4$ minors of a $5 \times 5$ symmetric matrix of indeterminates are a tropical basis.
\end{abstract}

The $r \times r$ minors of an $m \times n$ matrix of variables are a tropical basis if $r = 2, 3$, or $min(m,n)$. They are \emph{not} form a tropical basis if $4 < r < min(m,n)$. The $r = 4$ case is special. The $4 \times 4$ minors of an $m \times n$ matrix of variables are a tropical basis if $min(m,n) \leq 6$, but otherwise not.

The $r = 4$ case is exceptional for symmetric matrices as well. In this paper we prove the $4 \times 4$ minors of a symmetric $5 \times 5$ matrix of variables form a tropical basis, and develop a method that might generalize to larger symmetric matrices. The paper begins with a review of the basic concepts from tropical geometry that we will use. It then introduces a technique called "the method of joints", and uses it to prove the $4 \times 4$ minors of a symmetric $5 \times 5$ matrix of indeterminates form a tropical basis. The paper concludes with an explanation for why the $4 \times 4$ minors of an $n \times n$ symmetric matrix of indeterminates do not form a tropical basis for $n > 12$, and why the author conjectures they do when $n \leq 12$.

A note on notation. When denoting the element in row $i$ and column $j$ of a matrix these indices will be separated by a comma, so for example $A_{i,j}$ is element $(i,j)$ of the matrix $A$. The notation $A_{ij}$ refers to the submatrix formed from $A$ be removing row $i$ and column $j$. Unless stated otherwise, the columns and rows of a submatrix inherit their indices from the larger matrix. So, if $A$ is a $5 \times 5$ matrix the principal submatrix $A_{33}$ has columns and rows labeled sequentially $1,2,4,5$.

The author would like to thank the mathematics department of the University of Utah for support during the research for this paper, and in particular his advisor Aaron Bertram.

\section{Tropical preliminaries}

This section introduces the basic ideas from tropical geometry used in this paper, and reviews relevant results from general and symmetric matrices.

\subsection{Tropical basics}

The \emph{tropical semiring} $(\mathbb{R}, \oplus, \odot)$, is defined as the semiring with arithmetic operations:
\begin{center}
  $a \oplus b := min(a,b)$ \hspace{.1 in} and \hspace{.1 in} $a \odot b := a + b$.
\end{center}

A \emph{tropical monomial} $X_{1}^{a_{1}} \cdots X_{m}^{a_{m}}$ is a symbol, and represents a function equivalent to the linear form $\sum_{i} a_{i}X_{i}$ (standard addition and multiplication). 

A \emph{tropical polynomial} is a tropical sum of tropical monomials  
  \begin{center}
    $F(X_{1},\ldots,X_{m}) := \bigoplus_{a \in \mathcal{A}}C_{a}X_{1}^{a_{1}}X_{2}^{a_{2}} \cdots X_{m}^{a_{m}}$, \hspace{.1 in} with $\mathcal{A} \subset \mathbb{N}^{m}$, $C_{a} \in \mathbb{R}$
  \end{center}
  (tropical addition and multiplication), and represents a piecewise linear convex function $F: \mathbb{R}^{m} \rightarrow \mathbb{R}$. 

In this paper, tropical polynomials will be represented with upper case letters, while standard polynomials will be lower case.

The \emph{tropical hypersurface} $\textbf{V}(F)$ defined by a tropical polynomial $F$ is the locus of points $P \in \mathbb{R}^{m}$ such that at least two monomials in $F$ are minimal at $P$. This is also called the \emph{double-min locus} of $F$.

For example, the tropical hypersurface defined by the tropical polynomial
\begin{center}
  $X \oplus Y \oplus 0  = min\{x,y,0\}$
\end{center}
would include the point $(1,0)$, as both $Y$ and $0$ are minimal at that point, but would not include the point $(-1,0)$, as $X$ is uniquely minimal at that point. This is an example of a tropical line.

\vspace{.1 in}
\begin{tabular}{c}
  \centering
  \hspace{1in}\includegraphics[scale=1]{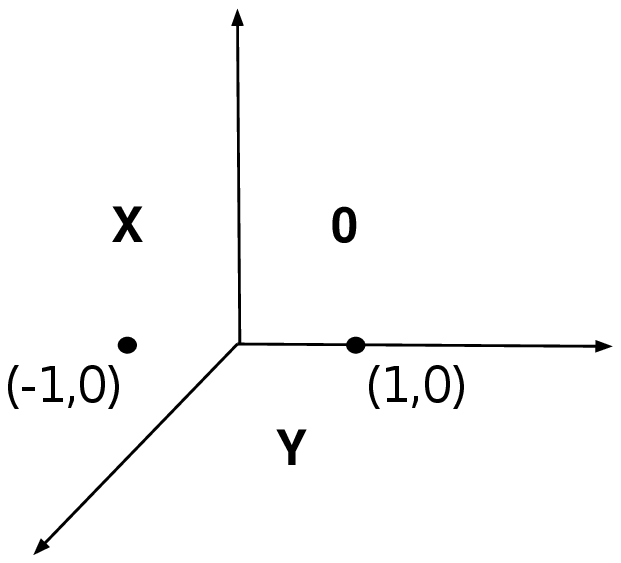}
\end{tabular}

\subsection{Tropical bases}

Let $k$ be an algebraically closed field. Let $f \in k[x_{1},\ldots,x_{m}]$ be a polynomial. The locus of points $p \in k^{m}$ such that $f(p) = 0$ is a \emph{hypersurface}, and is denoted $\textbf{V}(f)$. Let $I$ be an ideal of $k[x_{1},\ldots,x_{m}]$. The ideal $I$ defines a \emph{algebraic variety}, (or \emph{variety}, for short) $\textbf{V}(I)$, in $k^{m}$, which is the set of points $p \in k^{m}$ such that $f(p) = 0$ for all $f \in I$. If $I = (f_{1},\ldots,f_{n})$  then the set $\{f_{1},\ldots,f_{n}\}$ is a \emph{basis} for $I$, and $\textbf{V}(I)$ is equal to the locus of points $p \in k^{m}$ such that $f_{i}(p) = 0$ for all $f_{i}$ in the basis. Put succinctly
\begin{center}
  $\textbf{V}(I) = \bigcap \textbf{V}(f_{i})$.
\end{center}
So, a variety is an intersection of hypersurfaces. By the Hilbert basis theorem every ideal of $k[x_{1},\ldots,x_{m}]$ is finitely generated, so any variety is a finite intersection of hypersurfaces.

In the tropical setting there is an analog of a hypersurface, and we would like an analog of a variety. It might seem natural to define a tropical variety as the intersection of a finite set of tropical hypersurfaces, but these sets do not always have the properties we need in order for them to be useful analogs of algebraic varieties, and we instead call these sets tropical prevarieties. 

A \emph{tropical prevariety} $\textbf{V}(F_{1},\ldots,F_{n})$ is a finite intersection of tropical hypersurfaces:  
\begin{center}
  $\textbf{V}(F_{1},\ldots,F_{n}) = \bigcap_{i = 1}^{n} \textbf{V}(F_{i})$.
\end{center}

A tropical variety is defined differently. Let $K = \mathbb{C}\{\{t\}\}$ be the set of formal power series $a = c_{1}t^{a_{1}} + c_{2}t^{a_{2}} + \cdots$, where $a_{1} < a_{2} < a_{3} < \cdots$ are rational numbers that have a common denominator. These are called Puiseux series, and this set is an algebraically closed field of characteristic zero. For any nonzero element $a \in K$ define the degree of $a$ to be the value of the leading exponent $a_{1}$. This gives us a degree map $deg : K^{*} \rightarrow \mathbb{Q}$. For any two elements $a,b \in K^{*}$ we have
\begin{center}
  $deg(ab) = deg(a) + deg(b) = deg(a) \odot deg(b)$.
\end{center}
Generally, we also have
\begin{center}
  $deg(a + b) = min(deg(a),deg(b)) = deg(a) \oplus deg(b)$.
\end{center}
The only case when this addition relation is not true is when $a$ and $b$ have the same degree, and the coefficients of the leading terms cancel.

We would like to do tropical arithmetic over $\mathbb{R}$, and not just over $\mathbb{Q}$, so we enlarge the field of Puisieux series to allow this. Define the set $\tilde{K}$ by
\begin{center}
  $\tilde{K} = \left\{\sum_{\alpha \in A} c_{\alpha}t^{\alpha} | A \subset \mathbb{R} \text{ well-ordered}, c_{\alpha} \in \mathbb{C}\right\}$.
\end{center}
This is the set of Hahn series, and it is an algebraically closed field of characteristic zero containing the Puisieux series. We define a tropical variety in terms of a variety over $\tilde{K}$.

The degree map on $(\tilde{K}^{*})^{m}$ is the map $\mathcal{T}$ taking points $(p_{1},\ldots,p_{m}) \in (\tilde{K}^{*})^{m}$ to points $(deg(p_{1}),deg(p_{2}),\ldots,deg(p_{m})) \in \mathbb{R}^{m}$. A tropical variety is the image of a variety in $(\tilde{K}^{*})^{m}$ under the degree map. We call this image the \emph{tropicalization} of a set of points in $(\tilde{K}^{*})^{m}$. The tropicalization of a polynomial $f \in \tilde{K}[x_{1},\ldots,x_{m}]$ is the tropical polynomial $\mathcal{T}(f)$ formed by tropicalizing the coefficients of $f$, and converting addition and multiplication into their tropical counterparts. For example, the tropicalization of the polynomial 
\begin{center}
  $f = 3t^{2}xy - 7tx^{3}$ 
\end{center}
is the tropical polynomial
\begin{center}
  $\mathcal{T}(f) = 2XY \oplus 1X^{3}$.
\end{center}

In an unpublished manuscript, Mikhail Kapranov proved the following useful and fundamental result.

\newtheorem{thm}{Theorem}
\begin{thm}[Kapranov's Theorem]
  For $f \in \tilde{K}[x_{1},\ldots,x_{m}]$ the tropical variety $\mathcal{T}(\textbf{V}(f))$ is equal to the tropical hypersurface $\textbf{V}(\mathcal{T}(f))$ determined by the tropical polynomial $\mathcal{T}(f)$.
\end{thm}

Given Kapranov's theorem if $I = (f_{1},\ldots,f_{n})$, then obviously the tropical prevariety determined by the set of tropical polynomials $\{\mathcal{T}(f_{1}),\ldots,\mathcal{T}(f_{n})\}$ contains the tropical variety determined by $I$:
\begin{center}
  $\mathcal{T}(\textbf{V}(I)) \subseteq \bigcap_{i = 1}^{n} \textbf{V}(\mathcal{T}(f_{i}))$.
\end{center}

While Kapranov's theorem gives us the two sets are equal if $n = 1$, in general the containment may be strict. For example, the lines in $(\tilde{K}^{*})^{2}$ defined by the linear equations
\begin{center}
  $f = 2x + y + 1$, \hspace{.1 in} and \hspace{.1 in} $g = tx + ty + 1$,
\end{center}
intersect at the point $(t^{-1}-1,-2t^{-1}+1)$. The tropicalization of this point is $(-1,-1)$, and so if $I = (f,g)$ then
\begin{center}
  $\mathcal{T}(\textbf{V}(I)) = (-1,-1)$.
\end{center}
However, is we tropicalize the linear equations we get:
\begin{center}
  $\mathcal{T}(f) = X \oplus Y \oplus 0$, \hspace{.1 in} and \hspace{.1 in} $\mathcal{T}(g) = 1X \oplus 1Y \oplus 0$.
\end{center}
Each of $\textbf{V}(\mathcal{T}(f))$ and $\textbf{V}(\mathcal{T}(g))$ is a tropical line, and their intersection is the tropical prevariety consisting of all points $(a,a)$ with $a \leq -1$.

\vspace{.1 in}
\begin{tabular}{c}
  \centering
  \hspace{1in}\includegraphics[scale=1]{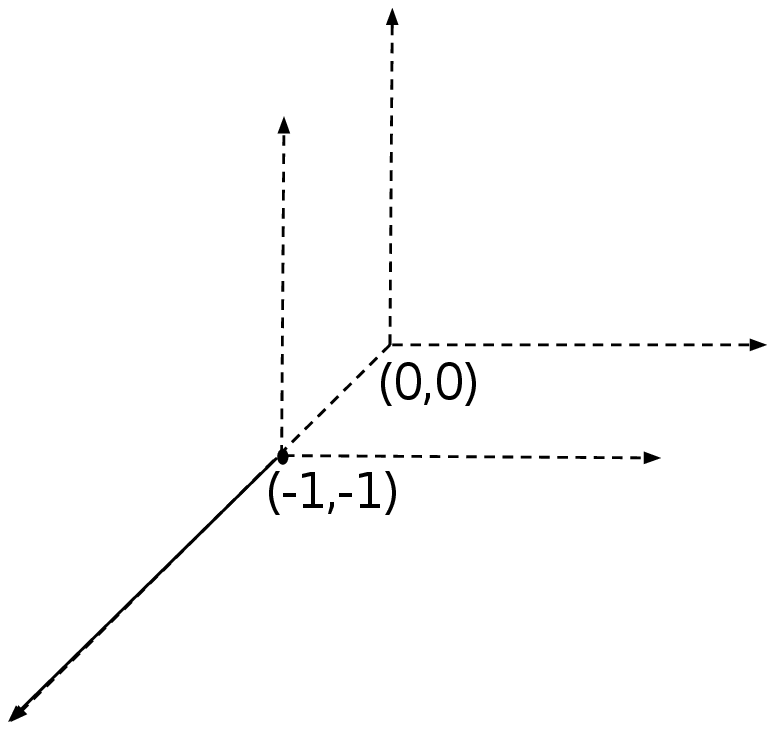}
\end{tabular}

This tropical prevariety properly contains the tropical variety $(-1,-1)$, but the prevariety is clearly much larger. That the intersection of two distinct tropical lines is not necessarily a point is a motivating example of why we do not define a tropical variety to be a finite intersection of tropical hypersurfaces.

\subsection{Kapranov and tropical Rank}

In \cite{dss}, Develin, Santos, and Sturmfels define three notions of matrix rank coming from tropical geometry: the Barvinok rank, the Kapranov rank, and the tropical rank. In this paper we focus on the symmetric analogs of the Kapranov and tropical ranks.

The \emph{tropical rank} of an $m \times n$ matrix $A = (A_{i,j}) \in \mathbb{R}^{m \times n}$ is the smallest number $r \leq min(m,n)$ such that $A$ is not in the tropical prevariety formed by the $r \times r$ minors of an $m \times n$ matrix of indeterminates.

A \emph{lift} of the matrix $A$ is a matrix $\tilde{A} = (\tilde{a}_{i,j}) \in (\tilde{K}^{*})^{m \times n}$ such that $deg(\tilde{a}_{i,j}) = A_{i,j}$ for all $i,j$. The \emph{Kapranov rank} of a matrix is the smallest rank of any lift of the matrix. Equivalently, the Kapranov rank is the smallest number $r \leq min(m,n)$ such that $A$ is not in the tropical variety formed by the $r \times r$ minors of an $m \times n$ matrix of indeterminates.

A square matrix $A = (A_{i,j}) \in \mathbb{R}^{n \times n}$ is \emph{tropically singular} if the minimum
  \begin{center}
    $tropdet(A) := \bigoplus_{\sigma \in S_{n}} A_{1,\sigma(1)} \odot A_{2,\sigma(2)} \odot \cdots \odot A_{n,\sigma(n)}$
  \end{center}
  is attained at least twice is the tropical sum. Here $S_{n}$ denotes the symmetric group on $\{1,2,\ldots,n\}$. We call the number $tropdet(A)$ defined above the \emph{tropical determinant} of $A$, and we say any permutation $\sigma$ such that
  \begin{center}
    $tropdet(A) = A_{1,\sigma(1)} \odot A_{2,\sigma(2)} \odot \cdots \odot A_{n,\sigma(n)}$
  \end{center}
  \emph{realizes} the tropical determinant. So, equivalently, a square matrix $A$ is tropically singular if more than one permutation realizes the tropical determinant.

More generally, suppose $A$ is an $m \times n$ real matrix, and $\{i_{1},i_{2},\ldots,i_{r}\}$ and $\{j_{1},j_{2},\ldots,j_{r}\}$ are subsets of $\{1,\ldots,m\}$ and $\{1,\ldots,n\}$, respectively. These subsets define an $r \times r$ submatrix $A'$ of $A$, with row indices $\{i_{1},\ldots,i_{r}\}$ and column indices $\{j_{1},\ldots,j_{r}\}$. A tropical monomial of the form
\begin{center}
  $\bigodot_{k = 1}^{r}X_{i_{k},\rho(i_{k})}$,
\end{center}
where $\rho$ is a bijection from the row indices to the column indices, is a \emph{minimizing monomial} for the submatrix $A'$ if, over all monomials defined by bijections from $\{i_{1},i_{2},\ldots,i_{r}\}$ to $\{j_{1},j_{2},\ldots,j_{r}\}$, this monomial is minimal under the valuation $X_{i,j} \mapsto A_{i,j}$. An $r \times r$ submatrix of $A$ is tropically singular if it has more than one minimizing monomial.

The tropical variety defined by the $r \times r$ minors of an $m \times n$ matrix of indeterminates is contained in the tropical prevariety defined by the same minors, and therefore
\begin{center}
  tropical rank $\leq$ Kapranov rank.
\end{center}

A natural question to ask about Kapranov and tropical rank is when they are necessarily equal. In other words, for what values $r,m,n$ does tropical rank $r$ imply Kapranov rank $r$ for any $m \times n$ matrix.

This question was answered through the combined work of Develin, Santos, and Sturmfels \cite{dss}, Chan, Jensen, and Rubei \cite{cjr}, and Shitov \cite{sh}. The result is named after Shitov \cite{ms}, as he completed the project.

\begin{thm}[Shitov's Theorem]
  The $r \times r$ minors of an $m \times n$ matrix of indeterminates form a tropical basis if and only if:
  \begin{itemize}
    \item $r = min(m,n)$, or
    \item $r \leq 3$, or
    \item $r = 4$ and $min(m,n) \leq 6$.
  \end{itemize}
\end{thm}

As with general matrices, for symmetric matrices the $r = 4$ minors are a special boundary case. 

\subsection{Symmetric Kapranov and symmetric tropical rank}

The symmetric Kapranov and symmetric tropcial ranks are defined analogously to their general counterparts.

The \emph{symmetric tropical rank} of an $n \times n$ symmetric matrix $A = (A_{i,j}) \in \mathbb{R}^{n \times n}$ is the smallest number $r \leq n$ such that $A$ is not in the tropical prevariety formed by the $r \times r$ minors of an $n \times n$ symmetric matrix of indeterminates.

The \emph{symmetric Kapranov rank} of an $n \times n$ symmetric matrix $A = (A_{i,j}) \in \mathbb{R}^{n \times n}$ is the smallest rank of any lift to a symmetric matrix. Equivalently, the symmetric Kapranov rank is the smallest number $r \leq n$ such that $A$ is not in the tropical variety formed by the $r \times r$ minors of an $n \times n$ symmetric matrix of indeterminates.

The tropical variety defined by the $r \times r$ minors of an $n \times n$ symmetric matrix of indeterminates is contained in the tropical prevariety defined by the same minors, and therefore
\begin{center}
  symmetric tropical rank $\leq$ symmetric Kapranov rank.
\end{center}

As a lift to an $n \times n$ symmetric matrix is a lift to an $n \times n$ matrix, we must have
\begin{center}
  Kapranov rank $\leq$ symmetric Kapranov rank.
\end{center}

For symmetric matrices, we say a square submatrix is \emph{symmetrically tropically singular} if it has more than one minimizing monomial given the equivalence $X_{i,j} = X_{j,i}$.

The tropical rank of a matrix can be equivalently defined as the smallest value of $r$ such that the matrix has a tropically nonsingular $r \times r$ submatrix, and similarly the symmetric tropical rank of a symmetric matrix can be equivalently defined as the smallest value of $r$ such that the symmetric matrix has a symmetrically tropically nonsingular $r \times r$ submatrix. 

If an $r \times r$ submatrix of a symmetric $n \times n$ matrix has two distinct minimizing monomials, given the equivalence $X_{i,j} = X_{j,i}$, then a fortiori it has two distinct minimizing monomials without that equivalence, and so
\begin{center}
  tropical rank $\leq$ symmetric tropical rank.
\end{center}

For example, the tropical determinant of a $3 \times 3$ matrix of indeterminates 
\begin{center}
  $\left(\begin{array}{ccc} X_{1,1} & X_{1,2} & X_{1,3} \\ X_{2,1} & X_{2,2} & X_{2,3} \\ X_{3,1} & X_{3,2} & X_{3,3} \end{array}\right)$
\end{center}
is
\begin{center}
  $X_{1,1}X_{2,2}X_{3,3} \oplus X_{1,2}X_{2,3}X_{3,1} \oplus X_{1,3}X_{2,1}X_{3,2} \oplus X_{1,1}X_{2,3}X_{3,2} \oplus X_{1,2}X_{2,1}X_{3,3} \oplus X_{1,3}X_{2,2}X_{3,1}$.
\end{center}
For the matrix
\begin{center}
    $\left(\begin{array}{ccc} 1 & 0 & 0 \\ 0 & 1 & 0 \\ 0 & 0 & 1 \end{array}\right)$
  \end{center}
there are two minimizing monomials, $X_{1,2}X_{2,3}X_{3,1}$ and $X_{1,3}X_{2,1}X_{3,2}$, and so the matrix is tropically singular. However, under the equivalence $X_{i,j} = X_{j,i}$ the tropical determinant becomes
\begin{center}
  $X_{1,1}X_{2,2}X_{3,3} \oplus X_{1,2}X_{2,3}X_{1,3} \oplus X_{1,1}X_{2,3}^{2} \oplus X_{1,2}^{2}X_{3,3} \oplus X_{1,3}^{2}X_{2,2}$,
\end{center}
and the monomial $X_{1,2}X_{2,3}X_{1,3}$ is the unique minimizing monomial, so the matrix is not symmetrically tropically singular.

In \cite{z} the author proves the following symmetric partial analog of Shitov's theorem.

\begin{thm}
    The $r \times r$ minors of an $n \times n$ symmetric matrix of indeterminates are a tropical basis if $r = 2$, $r = 3$, or $r = n$. The minors are not a tropical basis if $4 < r < n$, or if $r = 4$ and $n > 12$.
\end{thm}

The boundary cases $r = 4$, and $5 \leq n \leq 12$ remained a question in \cite{z}. In this paper, we partially answer that question for $n = 5$. 

\section{The $4 \times 4$ minors of a $5 \times 5$ symmetric matrix are a tropical basis}

Before we begin the major proof for this paper, we'll need some basic facts about modifying symmetric matrices which we will frequently use.

\begin{itemize}
  
    \item If $A$ is a symmetric matrix, and we permute the rows of $A$ by a permutation $\sigma$, and the columns of $A$ by the \emph{same} permutation, then the resulting matrix $A'$ will be symmetric, and $A'$ will have the same symmetric tropical and symmetric Kapranov rank as $A$. We call a permutation of the rows and columns of $A$ by the same permutation a \emph{diagonal permutation}. 
      
    \item If $A$ is a symmetric matrix, and we tropically multiply row $i$ by a constant $c$, and tropically multiply column $i$ by the \emph{same} constant, then the resulting matrix $A'$ will be symmetric, and $A'$ will have the same symmetric tropical and symmetric Kapranov rank as $A$. In fact, both $A$ and $A'$ will have the same minimizivg monomials, and so will any submatrix of $A$ and the corresponding submatrix of $A'$. We call such an operation a \emph{symmetric} scaling of $A$.
      
\end{itemize}

\newtheorem{lemma}{Lemma}
\begin{lemma}
    For a symmetric matrix if there is a permutation that realizes the tropical determinant, and this permutation contains a 4-cycle in its disjoint cycle decomposition, then the tropical determinant is also realized by a product of transpositions. More specifically, if the tropical determinant is realized by a permutation that contains $(k_{1}k_{2}k_{3}k_{4})$ in its disjoint cycle decomposition, then this cycle can be replaced by $(k_{1}k_{2})(k_{3}k_{4})$ or $(k_{1}k_{4})(k_{2}k_{3})$ to form another permutation that realized the tropical determinant. 
\end{lemma}

\begin{proof}
    The product of $X_{k_{1},k_{2}}X_{k_{2},k_{1}}X_{k_{3},k_{4}}X_{k_{4},k_{3}}$ and $X_{k_{1},k_{4}}X_{k_{4},k_{1}}X_{k_{2},k_{3}}X_{k_{3},k_{2}}$ is the square of $X_{k_{1},k_{2}}X_{k_{2},k_{3}}X_{k_{3},k_{4}}X_{k_{4},k_{1}}$, given the equivalence $X_{i,j} = X_{j,i}$, and the only way $X_{k_{1},k_{2}}X_{k_{2},k_{3}}X_{k_{3},k_{4}}X_{k_{4},k_{1}}$ could be a minimizing monomial is if $X_{k_{1},k_{2}}X_{k_{2},k_{1}}X_{k_{3},k_{4}}X_{k_{4},k_{3}}$ and $X_{k_{1},k_{4}}X_{k_{4},k_{1}}X_{k_{2},k_{3}}X_{k_{3},k_{2}}$ are as well. 
\end{proof}

This can be easily generalized to any odd-cycle larger that a transposition, but we will not need that generalization in this paper.

\newtheorem{proposition}{Proposition}
\begin{proposition}
  If $A$ is a $5 \times 5$ symmetric matrix and $\sigma$ is a permutation that realizes the tropical determinant, then there exists a matrix $A'$ such that $A'$ can be obtained from $A$ through a diagonal permutation and a sequence of symmetric scalings, every entry in $A'$ is nonnegative, and $a_{i,\sigma(i)} = 0$ for all $1 \leq i \leq n$.
\end{proposition}

\begin{proof}
  Note that within this proof, and only within this proof, if we are talking about the ``form'' of a matrix a blank entry can have any value, positive or negative. Also, note that when we scale row / column $i$ by an amount $c$, that means row $i$ is tropically multiplied by $c$, and column $i$ is also tropically multiplied by $c$, so, for example, the element $a_{i,i}$ is tropically multiplied by $c$ twice.

  If $\sigma = id$ then we can form $A'$ by scaling each row/column $i$ by $-a_{i,i}/2$ to obtain a matrix with the form
  \begin{center}
    $\left(\begin{array}{ccccc} 0 & & & & \\ & 0 & & & \\ & & 0 & & \\ & & & 0 & \\ & & & & 0 \end{array}\right)$.
  \end{center}
  The tropical determinant must be realized by $\sigma = id$, and the matrix must be symmetric. If any off-diagonal element, and its symmetric counterpart were negative, the tropical determinant would be negative, which would be a contradiction. So, all off-diagonal elements must be nonnegative, and $A'$ has the desired properties.

  If $\sigma$ is a transposition ($2$-cycle) then we can assume without loss of generalitiy that $\sigma = (12)$. Scale row / column $1$ by $-a_{1,2}/2$, and row / column $2$ by the same amount. Then, scale row / column $3$ by $-a_{3,3}/2$, and follow the same approach for row / column 4 and row / column 5. The matrix obtained from this scaling must have the form
  \begin{center}
    $\left(\begin{array}{ccccc} & 0 & & & \\ 0 & & & & \\ & & 0 & & \\ & & & 0 & \\ & & & & 0 \end{array}\right)$.
  \end{center}
  If $a_{1,1}$ is negative and $a_{2,2} < -a_{1,1}$, then the tropical determinant would be negative. If $a_{1,1}$ is negative and $a_{2,2} \geq -a_{1,1}$, then if we scale row / column $1$ by $-a_{1,1}/2$, and row / column $2$ by $a_{1,1}/2$, then all diagonal terms are non-negative. If any off-diagonal element and its symmetric counterpart are negative then if the element were in the bottom-right $3 \times 3$ submatrix the tropical determinant would be negative. If the element is in the top-right $2 \times 3$ submatrix, we may assume, possibly after a diagonal permutation, that it's element $a_{1,3}$, and this in the minimal element of the matrix. If $a_{2,2} < -2a_{1,3}$ then the tropical determinant would be negative. If we scale row / column $1$ by $-a_{1,3}$, and row / column $2$ by $a_{1,3}$, then the matrix maintains the same form above, with $a_{1,3} = 0$, and all the diagonal elements non-negative. This can be repeated until there are no negative elements in the top-right $2 \times 3$ submatrix. At this point, we have our desired matrix $A'$.

  If $\sigma$ is a $3$-cycle we can assume without loss of generality that $\sigma = (123)$. Scale row / column $2$ by $-a_{1,2}$, and row / column $3$ by $-a_{2,3}$. Then, scale rows / columns $1$ and $3$ by $-a_{1,3}/2$, and row / column $2$ by $a_{1,3}/2$. Scale row / column $4$ by $-a_{4,4}/2$, and row / column $5$ by $-a_{5,5}/2$. This scaled matrix will have the form
  \begin{center}
    $\left(\begin{array}{ccccc} & 0 & 0 & & \\ 0 & & 0 & & \\ 0 & 0 & & & \\ & & & 0 & \\ & & & & 0 \end{array}\right)$,
  \end{center}
  and its tropical determinant must be $0$. As in the previous example, if any blank entry, and its symmetric counterpart, were negative the matrix would have negative determinant. So, we have our desired matrix $A'$.

  If $\sigma$ is the product of two transpositions we can assume without loss of generality that $\sigma = (12)(34)$. Scale row / column $1$ by $-a_{1,2}/2$ and row / column $2$ by $-a_{1,2}/2$. If after this scaling either of the top two diagonal terms are negative we can scale as we did in the $\sigma = (12)$ case to keep the off-diagonal terms $0$ and make the diagonal terms nonnegative. The same can be done for the $2 \times 2$ block corresponding with the transposition $(34)$. Scale row /column $5$ by $-a_{5,5}/2$ to get the matrix with form
  \begin{center}
    $\left(\begin{array}{ccccc} & 0 & & & \\ 0 & & & & \\ & & & 0 & \\ & & 0 & & \\ & & & & 0 \end{array}\right)$.
  \end{center}
  If any element in the submatrix defined by rows 1 and 2, and columns 3 and 4, is negative, we may assume, possibly after a diagonal permutation, that it is $a_{2,3}$, and it is minimal over elements in that submatrix. If $-2a_{2,3} > a_{1,1}+a_{4,4}$ then the tropical determinant would be negative. We can scale row / column $2$ by $a_{1,1}/2$, and row / column $3$ by $a_{4,4}/2$. We can then scale row / column 1 by $-a_{1,1}/2$ and row / column 4 by $-a_{4,4}/2$. This would maintain the form of the matrix, maintain the diagonal elements as non-negative, and make $a_{2,3}$ non-negative. This can be repeated until there are no negative elements in the submatrix defined by rows 1 and 2, and columns 3 and 4. If any element in the submatrix defined by rows 1, 2, 3, 4 and column 5 is negative, we may assume, possibly after a diagonal permutation, that it is $a_{1,5}$, and it is minimal over elements in that submatrix. If $-2a_{1,5} > a_{2,2}$, then the tropical determinant would be negative. We can scale row / column 1 by $-a_{1,5}$, and row / column 2 by $a_{1,5}$. This maintains the form of the matrix, maintains the diagonal elements as non-negative, and makes $a_{1,5}$ nonnegative. This can be repeated until there are no negative elements in the submatrix defined by rows 1, 2, 3, 4 and column 5. This gives us our desired matrix $A'$.

  If $\sigma$ is a 4-cycle, by Lemma 1 we can instead assume it is the product of two transpositions handled above.

  If $\sigma$ is a $3$-cycle and a tranposition we can assume without loss of generality that $\sigma = (123)(45)$. We can scale the first three indices exactly as we did in the $\sigma = (123)$ case. If we then scale both rows / columns $4$ and $5$ by $-a_{4,5}/2$ we construct a matrix of the form
  \begin{center}
    $\left(\begin{array}{ccccc} & 0 & 0 & & \\ 0 & & 0 & & \\ 0 & 0 & & & \\ & & & & 0 \\ & & & 0 & \end{array}\right)$.
  \end{center}
  where the tropical determinant is $0$. If any entry along the top three diagonal terms were negative the determinant of the matrix would be negative. If $a_{4,4} < 0$, then $a_{5,5} \geq -a_{4,4}$ given the tropical determinant is $0$. We can scale row / column $4$ by $-a_{4,4}/2$, and row / column $5$ by the opposite amount. This keeps the matrix in the form above, but with $0$ in the $(4,4)$ entry. Exactly the same reasoning applies if $a_{5,5} < 0$. If any other entry were negative we can assume without loss of generality that the minimum entry in the matrix is $a_{3,4}$ and its symmetric counterpart $a_{4,3}$. If $a_{5,5} < -2a_{3,4}$ the matrix would have negative determinant, which would be a contradiction. If we scale row / column $4$ by $-a_{3,4}''$, and row / column $5$ by the opposite, then the matrix maintains the form above, but with $a_{3,4}$ nonnegative. This can be repeated until there are no negative elements, and we obtain our matrix $A'$.

  If $\sigma$ is a $5$-cycle then we can assume without loss of generality that $\sigma = (12345)$. Scale row / column $2$ by $-a_{1,2}/2$, row / column $3$ by $-a_{2,3}/2$, and so on until row / column $5$. Next, scale all the rows / columns with odd labels by an amount equal to $-a_{1,5}/2$, and scale all the rows / columns with even labels by an amount equal to $a_{1,5}/2$. The matrix obtained from this scaling must have the form
  \begin{center}
    $\left(\begin{array}{ccccc} & 0 & & & 0 \\ 0 & & 0 & & \\ & 0 & & 0 & \\ & & 0 & & 0 \\ 0 & & & 0 & \end{array}\right)$,
  \end{center}
  and its tropical determinant must be realized by $\sigma = (12345)$, which means its tropical determinant must be $0$. If any blank entry, and its symmetric counterpart, were negative the tropical determinant of the matrix would be negative. So, we have our matrix $A'$.

  This exhausts all possibilities, and our proposition is proven.
\end{proof}

In this section from here on we will assume without loss of generality that all $5 \times 5$ matrices have been symmetrically scaled to satisfy the properties of Proposition 1.

We will frequently want to deal with all matrices that have a certain structure, and this structure will be captured by the \emph{form} of the matrix.

A \emph{form matrix} is a matrix in which every entry is either blank, a non-negative constant, or the symbol $''+''$. A nonnegative matrix $A$ has the \emph{form} of a form matrix $A'$ if everywhere $A'$ has a constant, $A$ has the same constant, and everywhere $A'$ has a $''+''$, $A$ has a positive entry.

For example, the matrix
\begin{center}
  $\left(\begin{array}{ccc} 0 & 2 & 1 \\ 2 & 0 & 3 \\ 1 & 3 & 0 \end{array}\right)$
\end{center}
\vspace{.1 in}
has any of the following forms:
\begin{center}
  $\left(\begin{array}{ccc} 0 & + & + \\ + & 0 & + \\ + & + & 0 \end{array}\right)$, \hspace{.1 in} $\left(\begin{array}{ccc} 0 & 2 & \\ 2 & 0 & + \\ & + & 0 \end{array}\right)$, \hspace{.1 in} $\left(\begin{array}{ccc} & + & + \\ + & 0 & + \\ + & + & 0 \end{array}\right)$, \hspace{.1 in} $\left(\begin{array}{ccc} 0 & 2 & 1 \\ 2 & 0 & 3 \\ 1 & 3 & 0 \end{array}\right)$.
\end{center}
It does not, however, have the form
\begin{center}
  $\left(\begin{array}{ccc} + & + & + \\ + & 0 & + \\ + & + & 0 \end{array}\right)$,
\end{center}
because it has a $0$ as its upper-left entry.

\subsection{The method of joints}

We now define the ``method of joints'', which will be the primary method by which we prove the main theorem in this paper.

Suppose $A$ is a symmetric matrix, and there are distinct indices $i$ and $j$ (assume without loss of generality $i < j$) such that:
\begin{itemize}
    \item The principal submatrix $A_{ii}$ is symmetrically tropically singular, and there are distinct minimizing monomials $X_{\sigma_{1}}, X_{\sigma_{2}}$ of $A_{ii}$, such that the variables in $X_{\sigma_{1}}$ involving the index $j$ are not the same as the variables in $X_{\sigma_{2}}$ involving the index $j$.
    \item The same is true with $i$ and $j$ reversed.
    \item The submatrix $A_{ji}$ is symmetrically tropically singular, and there are two minimizing monomials $X_{\tau_{1}}, X_{\tau_{2}}$ of $A_{ji}$ such that $X_{\tau_{1}}$ contains the variable $X_{i,j}$, while $X_{\tau_{2}}$ does not.
\end{itemize}
In this case, the indices $i$ and $j$ are called \emph{joints} of the matrix $A$. If the submatrix $A_{ii}$ satisfies the first condition above, we say it \emph{satisfies the joint requirement} for joints $i$ and $j$. Similarly for the submatrix $A_{jj}$.

For example, consider a matrix $A$ of the form
\begin{center}
  $\left(\begin{array}{ccccc}  & 0 & & & \\ 0 & & & & \\ & & 0 & 0 & 0 \\ & & 0 & 0 & 0 \\ & & 0 & 0 & 0 \end{array}\right)$.
\end{center}
We will demonstrate this matrix has joints $4$ and $5$. 

The principal submatrix $A_{44}$ has the form
\begin{center}
  $\left(\begin{array}{cccc} & 0 & & \\ 0 & & & \\ & & 0 & 0 \\ & & 0 & 0 \end{array}\right)$.
\end{center}
This submatrix is symmetrically tropically singular, with minimizing monomials $X_{1,2}^{2}X_{3,3}X_{5,5}$ and $X_{1,2}^{2}X_{3,5}^{2}$. In particular, the only variable in the first monomial involving the index $5$ is $X_{5,5}$, while the second monomial contains the variable $X_{3,5}$. So, $A_{44}$ satisfies the joint requirement for joints $4$ and $5$. Identical reasoning can be applied to the principal submatrix $A_{55}$. 

The submatrix $A_{54}$ has the form
\begin{center}
  $\left(\begin{array}{cccc} & 0 & & \\ 0 & & & \\ & & 0 & 0 \\ & & 0 & 0 \end{array}\right)$.
\end{center}
The submatrix is symmetrically tropically singular, with minimizing monomials $X_{1,2}^{2}X_{3,3}X_{4,5}$ and $X_{1,2}^{2}X_{3,4}X_{3,5}$. One of these minimizing monomials contains the variable $X_{4,5}$, while the other does not. Therefore, $A$ has joints $4$ and $5$.

Our proof that the $4 \times 4$ minors of a symmetric $5 \times 5$ matrix are a tropical basis rests upon first proving that every symmetric matrix over $\mathbb{R}$ with joints has symmetric Kapranov rank of at most three. We then prove an exceptional case of a $5 \times 5$ symmetric matrix over $\mathbb{R}$ that does not have joints, but still has symmetric Kapranov rank three. Finally, we prove that if the $4 \times 4$ submatrices of a $5 \times 5$ symmetric matrix are all symmetrically tropically singular then either $A$ has joints, or $A$ has the form of the exceptional case. 

\begin{proposition}
  If a $5 \times 5$ symmetric matrix $A$ has joints, then it has symmetric Kapranov rank at most three.
\end{proposition}

\begin{proof}
  We will construct a symmetric rank three lift $\tilde{A}$ of $A$. After possibly a diagonal permutation we may assume $A$ has joints $4$ and $5$. We define the matrices:
  \begin{center}
    $X_{55} := \left(\begin{array}{cccc} A_{1,1} & A_{1,2} & A_{1,3} & X_{1,4} \\ A_{1,2} & A_{2,2} & A_{2,3} & X_{2,4} \\ A_{1,3} & A_{2,3} & A_{3,3} & X_{3,4} \\ X_{1,4} & X_{2,4} & X_{3,4} & X_{4,4} \end{array}\right)$,
  \end{center}
  
  and
  
  \begin{center}
    $\tilde{X}_{55} = \left(\begin{array}{cccc} a_{1,1} & a_{1,2} & a_{1,3} & x_{1,4} \\ a_{1,2} & a_{2,2} & a_{2,3} & x_{2,4} \\ a_{1,3} & a_{2,3} & a_{3,3} & x_{3,4} \\ x_{1,4} & x_{2,4} & x_{3,4} & x_{4,4}\end{array}\right)$,
  \end{center}
  where the $A_{i,j}$ are the same as the corresponding terms in the matrix $A$, and the $a_{i,j}$ terms are constants in the field $\tilde{K}$ such that $deg(a_{i,j}) = A_{i,j}$, but are otherwise generic. As the $a_{i,j}$ are generic, the tropicalization of the determinant of $\tilde{X}_{55}$ is the tropical determinant of $X_{55}$.
  
  By Kapranov's theorem if $(A_{1,4},A_{2,4},A_{3,4},A_{4,4})$ is a point on the tropical hypersurface given by the tropical determinant of $X_{55}$, then there is a lift to a point $(a_{1,4},a_{2,4},a_{3,4},a_{4,4})$ in $\tilde{K}^{4}$ on the hypersurface given by the determinant of $\tilde{X}_{55}$. This lift gives us a singular $4 \times 4$ matrix
  \begin{center}
    $\tilde{A}_{55} := \left(\begin{array}{cccc} a_{1,1} & a_{1,2} & a_{1,3} & a_{1,4} \\ a_{1,2} & a_{2,2} & a_{2,3} & a_{2,4} \\ a_{1,3} & a_{2,3} & a_{3,3} & a_{3,4} \\ a_{1,4} & a_{2,4} & a_{3,4} & a_{4,4} \end{array}\right)$,
  \end{center}
  that tropicalizes to the submatrix $A_{55}$ of $A$. An identical argument can be used to construct a singular lift of $A_{44}$
  \begin{center}
    $\tilde{A}_{44} = \left(\begin{array}{cccc} a_{1,1} & a_{1,2} & a_{1,3} & a_{1,5} \\ a_{1,2} & a_{2,2} & a_{2,3} & a_{2,5} \\ a_{1,3} & a_{2,3} & a_{3,3} & a_{3,5} \\ a_{1,5} & a_{2,5} & a_{3,5} & a_{5,5} \end{array}\right)$,
  \end{center}
  where the top-left $3 \times 3$ submatrics of $\tilde{A}_{44}$ and $\tilde{A}_{55}$ are identical. 
  
  We note that if we multiply the fourth column and the fourth row of $\tilde{A}_{55}$ by the same degree zero generic constant that we will still have a singular symmetric lift of $A_{55}$. So, we can assume the terms $a_{i,4}$ are generic relative to the terms $a_{j,5}$ for any $i,j \leq 5$, except for $a_{4,5}$ and $a_{5,4}$, which we have not yet determined, and which must, of course, be equal.
  
  All the entries in a lift of $A$ have now been determined except $a_{4,5} = a_{5,4}$. To get $a_{4,5}$ we examine the matrices:
  \begin{center}
    $X_{54} := \left(\begin{array}{cccc} A_{1,1} & A_{1,2} & A_{1,3} & A_{1,5} \\ A_{1,2} & A_{2,2} & A_{2,3} & A_{2,5} \\ A_{1,3} & A_{2,3} & A_{3,3} & A_{3,5} \\ A_{1,4} & A_{2,4} & A_{3,4} & X_{4,5} \end{array}\right)$,
  \end{center}
  
  and
  
  \begin{center}
    $\tilde{X}_{54} := \left(\begin{array}{cccc} a_{1,1} & a_{1,2} & a_{1,3} & a_{1,5} \\ a_{1,2} & a_{2,2} & a_{2,3} & a_{2,5} \\ a_{1,3} & a_{2,3} & a_{3,3} & a_{3,5} \\ a_{1,4} & a_{2,4} & a_{3,4} & x_{4,5} \end{array}\right)$.
  \end{center}
  The determinant of $\tilde{X}_{54}$ is a linear function in the variable $x_{4,5}$, and the tropical determinant of $X_{54}$ is a tropical linear function in the variable $X_{4,5}$. As the terms in the upper-left $3 \times 3$ submatrix of $\tilde{X}_{54}$ are generic, and the constant terms in the rightmost column of $\tilde{X}_{54}$ are generic with respect to the constant terms in the bottom row, the tropicalization of the determinant of $\tilde{X}_{54}$ is the determinant of $X_{54}$.
  
  Again, by Kapranov's theorem, if $A_{4,5}$ is on the tropical hypersurface given by the tropical determinant of $X_{54}$, then it lifts to a point on the determinant of $\tilde{X}_{54}$. In other words, if the tropical determinant of the matrix
  \begin{center}
    $\left(\begin{array}{cccc} A_{1,1} & A_{1,2} & A_{1,3} & A_{1,5} \\ A_{1,2} & A_{2,2} & A_{2,3} & A_{2,5} \\ A_{1,3} & A_{2,3} & A_{3,3} & A_{3,5} \\ A_{1,4} & A_{2,4} & A_{3,4} & A_{4,5} \end{array}\right)$
  \end{center}
  is realized by two minimizing monomials, one involving the variable $X_{4,5}$ and the other not, then there exists a value $a_{4,5} \in \tilde{K}$ that makes the matrix
  \begin{center}
    $\left(\begin{array}{cccc} a_{1,1} & a_{1,2} & a_{1,3} & a_{1,5} \\ a_{1,2} & a_{2,2} & a_{2,3} & a_{2,5} \\ a_{1,3} & a_{2,3} & a_{3,3} & a_{3,5} \\ a_{1,4} & a_{2,4} & a_{3,4} & a_{4,5} \end{array}\right)$
  \end{center}
  a singular lift of $A_{54}$.
  
  \vspace{.1 in}
  
  The requirements for our three applications of Kapranov's theorem are exactly the requirements that $4$ and $5$ are joints of $A$. So, if $A$ has joints $4$ and $5$ then we have now determined all the elements in a lift of the matrix $A$: 
  \begin{center}
    $\tilde{A} := \left(\begin{array}{ccccc} a_{1,1} & a_{1,2} & a_{1,3} & a_{1,4} & a_{1,5} \\ a_{1,2} & a_{2,2} & a_{2,3} & a_{2,4} & a_{2,5} \\ a_{1,3} & a_{2,3} & a_{3,3} & a_{3,4} & a_{3,5} \\ a_{1,4} & a_{2,4} & a_{3,4} & a_{4,4} & a_{4,5} \\ a_{1,5} & a_{2,5} & a_{3,5} & a_{4,5} & a_{5,5} \end{array}\right)$.
  \end{center}
  It remains to be proven that such a lift has rank three. We do this by first proving there is a linear combination of the first three columns equal to the fourth. As the entries in the upper-left $3 \times 3$ submatrix were chosen generically, this submatrix has rank three, and therefore there is a unique set of coefficients $c_{1},c_{2},c_{3} \in \tilde{K}$ such that
  \begin{center}
    $c_{1}\left(\begin{array}{c} a_{1,1} \\ a_{1,2} \\ a_{1,3} \end{array}\right) + c_{2}\left(\begin{array}{c} a_{1,2} \\ a_{2,2} \\ a_{2,3} \end{array}\right) + c_{3}\left(\begin{array}{c} a_{1,3} \\ a_{2,3} \\ a_{3,3} \end{array}\right) = \left(\begin{array}{c} a_{1,4} \\ a_{2,4} \\ a_{3,4} \end{array}\right)$.
  \end{center}
  That this unique set of coefficients also satisfy
  \begin{center}
    $c_{1}a_{1,4} + c_{2}a_{2,4} + c_{3}a_{3,4} = a_{4,4}$,
  \end{center}
  and
  \begin{center}
    $c_{1}a_{1,5} + c_{2}a_{2,5} + c_{3}a_{3,5} = a_{4,5}$
  \end{center}
  follows immediately from the singularity of $\tilde{A}_{45}$ and $\tilde{A}_{55}$, respectively. Identical reasoning proves that there exists a linear combination of the first three columns of $\tilde{A}$ equal to the fifth, using the singularity of $\tilde{A}_{54}$ (which, as it is the transpose of $\tilde{A}_{45}$, follows from the singularity of $\tilde{A}_{45}$) and $\tilde{A}_{44}$. Therefore $\tilde{A}$ is a rank three lift of $A$, and so $A$ has symmetric Kapranov rank at most three.
\end{proof}

\subsection{The exceptional case}

In our analysis of $5 \times 5$ symmetric matrices with symmetric tropical rank three or less, there is one possible form that does not have joints, but which still has symmetric Kapranov rank three.

\begin{proposition}
  If a symmetric tropical matrix $A$ has the form:   
  \begin{center}
    $\left(\begin{array}{ccccc} 0 & 0 & + & + & N_{1} \\ 0 & 0 & + & + & N_{2} \\ + & + & 0 & 0 & P \\ + & + & 0 & 0 & P \\ N_{1} & N_{2} & P & P & 0 \end{array}\right)$,
  \end{center}
  with $N_{1},P > 0$, $N_{2} \geq N_{1}$, and $N_{1} \otimes P$ less than any element in the $2 \times 2$ submatrix determined by rows 1 and 2, and columns 3 and 4, then $A$ has symmetric Kapranov rank three.
\end{proposition}

\begin{proof}
    The principal submatrix formed from the columns and rows with indices $1$, $3$, and $5$ has the form
    \begin{center}
        $\left(\begin{array}{ccc} 0 & + & N_{1} \\ + & 0 & P \\ N_{1} & P & 0 \end{array}\right)$.
    \end{center}
    Any matrix with this form is symmetrically tropically nonsingular, and therefore $A$ must have symmetric tropical rank at least three. Consequently, its symmetric Kapranov rank must be at least three.
  
    To prove $A$ has symmetric Kapranov rank exactly three, we first augment the matrix $A$, producing a matrix $A'$ with the form
    \begin{center}
        $\left(\begin{array}{cccccc} 0 & 0 & 0 & + & + & N_{1} \\ 0 & 0 & 0 & + & + & N_{2} \\ 0 & 0 & 0 & P & P & 0 \\ + & + & P & 0 & 0 & P \\ + & + & P & 0 & 0 & P \\ N_{1} & N_{2} & P & P & 0 & 0 \end{array}\right)$,
    \end{center}
    where $A_{33}' = A$. If $A'$ has a lift $\tilde{A}'$ to a symmetric rank three matrix, then $\tilde{A}_{33}'$ will be a symmetric rank three lift of $A$. So, it is sufficient to prove that $A'$ has symmetric Kapranov rank three.
  
    The upper-right $4 \times 4$ submatrix of $A'$ is tropically singular, and therefore has a lift to a singular $4 \times 4$ matrix:
    \begin{center}
        $\left(\begin{array}{cccc} a_{1,3} & a_{1,4} & a_{1,5} & a_{1,6} \\ a_{2,3} & a_{2,4} & a_{2,5} & a_{2,6} \\ a_{3,3} & a_{3,4} & a_{3,5} & a_{3,6} \\ a_{4,3} & a_{4,4} & a_{4,5} & a_{4,6} \end{array}\right)$.
    \end{center}
    As $deg(a_{3,4}) = deg(a_{4,3})$ we can multiply the first column of this matrix by a degree zero constant, $a_{3,4}/a_{4,3}$, so that in the new matrix the $(3,4)$ entry and the $(4,3)$ entry are equal, and the matrix is still singular. So, we can assume the singular lift has $a_{3,4} = a_{4,3}$, and use it construct a lift for columns $3$ through $6$ of $A'$:
    \begin{center}
        $\left(\begin{array}{cccc} a_{1,3} & a_{1,4} & a_{1,5} & a_{1,6} \\ a_{2,3} & a_{2,4} & a_{2,5} & a_{2,6} \\ a_{3,3} & a_{3,4} & a_{3,5} & a_{3,6} \\ a_{3,4} & a_{4,4} & a_{4,5} & a_{4,6} \\ a_{3,5} & a_{4,5} & a_{5,5} & a_{5,6} \\ a_{3,6} & a_{4,6} & a_{5,6} & a_{6,6} \end{array}\right)$
    \end{center}
    where $a_{5,5}, a_{5,6}$, and $a_{6,6}$ have not yet been determined. We know there is a linear combination of columns $\textbf{a}_{3}, \textbf{a}_{4}$, and $\textbf{a}_{6}$ (the third, fourth, and sixth columns of $\tilde{A}'$) such that for rows $1$ through $4$:
    \begin{center}
        $\alpha\textbf{a}_{3} + \beta\textbf{a}_{4} + \gamma\textbf{a}_{6} = \textbf{a}_{5}$.
    \end{center}
  
   If $deg(\alpha)$ were minimal out of $deg(\alpha), deg(\beta)$, and $deg(\gamma)$, then in order for the linear relation above to hold on the third row we would need either $deg(\alpha) = P$ or $deg(\alpha) = deg(\gamma) \leq P$. In the first case the linear relation on the fourth row would be impossible. Define $M$ to be the minimum element in the $2 \times 2$ submatrix of $A$ determined by rows $1$ and $2$, and columns $3$ and $4$. In the second case, given $P \otimes N_{1} < M$, the linear relation on the first row would be impossible. So, $deg(\alpha)$ cannot be minimal.
  
  If $deg(\gamma)$ were minimal, then given $deg(\alpha)$ is not also minimal, for the linear relation on the third row to work we would need $deg(\gamma) = P$. This would make the linear relation on the fourth row impossible.
  
  So, $deg(\beta)$ must be uniquely minimal. In order for the linear relation on the fourth row to work out we must have $deg(\beta) = 0$, and in order for the linear relation on the third row to work out we must have $deg(\alpha), deg(\gamma) \geq P$. If $deg(\alpha) = P$, then the linear relation on the second row would be impossible. So, $deg(\alpha) > P$. 
  
  Pick $a_{6,6}$ such that $deg(a_{6,6}) = 0$, but otherwise generically. With this the linear relations define $a_{5,6}$ and $a_{5,5}$ as 
  \begin{center}
    $a_{5,6} = \alpha a_{3,6} + \beta a_{4,6} + \gamma a_{6,6}$,
    
    $a_{5,5} = \alpha a_{3,5} + \beta a_{4,5} + \gamma a_{5,6}$.
  \end{center}
  Given the requirements on the degrees of $\alpha, \beta, \gamma$, the known degrees of the terms from the lift of the upper-right $4 \times 4$ submatrix of $A'$, and that $a_{6,6}$ has degree $0$ but is otherwise generic, we must have $deg(a_{5,6}) = P$, and $deg(a_{5,5}) = 0$. 
  
  What remains is to find values for $a_{1,1}$, $a_{1,2}$, and $a_{2,2}$ such that the matrix
  \begin{center}
    $\left(\begin{array}{cccccc} a_{1,1} & a_{1,2} & a_{1,3} & a_{1,4} & a_{1,5} & a_{1,6} \\ a_{1,2} & a_{2,2} & a_{2,3} & a_{2,4} & a_{2,5} & a_{2,6} \\ a_{1,3} & a_{2,3} & a_{3,3} & a_{3,4} & a_{3,5} & a_{3,6} \\ a_{1,4} & a_{2,4} & a_{3,4} & a_{4,4} & a_{4,5} & a_{4,6} \\ a_{1,5} & a_{2,5} & a_{3,5} & a_{4,5} & a_{5,5} & a_{5,6} \\ a_{1,6} & a_{2,6} & a_{3,6} & a_{4,6} & a_{5,6} & a_{6,6} \end{array}\right)$
  \end{center}
  has rank three and tropicalizes to $A'$. If we examine the submatrix formed by columns $1, 3, 4$ and $6$,
  \begin{center}
    $\left(\begin{array}{cccc} a_{1,1} & a_{1,3} & a_{1,4} & a_{1,6} \\ a_{1,2} & a_{2,3} & a_{2,4} & a_{2,6} \\ a_{1,3} & a_{3,3} & a_{3,4} & a_{3,6} \\ a_{1,4} & a_{3,4} & a_{4,4} & a_{4,6} \\ a_{1,5} & a_{3,5} & a_{4,5} & a_{5,6} \\ a_{1,6} & a_{3,6} & a_{4,6} & a_{6,6} \end{array}\right)$,
  \end{center}
  then we note that, as there is a linear combination of columns $\textbf{a}_{3}, \textbf{a}_{4}, \textbf{a}_{6}$ equal to column $\textbf{a}_{5}$ there is linear combination of rows $3$ $4$, and $6$ in the above $6 \times 4$ matrix equal to row $5$. We pick $a_{1,1}$ so that the matrix
  \begin{center}
    $\left(\begin{array}{cccc} a_{1,1} & a_{1,3} & a_{1,4} & a_{1,6} \\ a_{1,3} & a_{3,3} & a_{3,4} & a_{3,6} \\ a_{1,4} & a_{3,4} & a_{4,4} & a_{4,6} \\ a_{1,6} & a_{3,6} & a_{4,6} & a_{6,6} \end{array}\right)$
  \end{center}
  is singular. Given the known degrees of the elements in the matrix, and that $a_{6,6}$ is generic, we must have $deg(a_{1,1}) = 0$. We can use an identical method to construct $a_{1,2}$ of the appropriate degree. Therefore every row of the above $6 \times 4$ matrix can be constructed from rows $3$, $4$, and $6$, and so the matrix has rank three. In particular, this means the third column of $\tilde{A}'$ can be constructed as a linear combination of the first, fourth, and sixth columns.
  
  What remains to be proven is that $a_{2,2}$ can be chosen with the appropriate degree so that the second column of $\tilde{A}'$ can be written as a linear combination of the first, fourth, and sixth columns. To do this we examine the $6 \times 4$ submatrix
  \begin{center}
    $\left(\begin{array}{cccc} a_{1,1} & a_{1,2} & a_{1,4} & a_{1,6} \\ a_{1,2} & a_{2,2} & a_{2,4} & a_{2,6} \\ a_{1,3} & a_{2,3} & a_{3,4} & a_{3,6} \\ a_{1,4} & a_{2,4} & a_{4,4} & a_{4,6} \\ a_{1,5} & a_{2,5} & a_{4,5} & a_{5,6} \\ a_{1,6} & a_{2,6} & a_{4,6} & a_{6,6} \end{array}\right)$.
  \end{center}
  We already know rows $3$ and $5$ of this matrix can be written as a linear combination of rows $1$, $4$, and $6$. To prove this is also true for row $2$ we examine the submatrix
  \begin{center}
    $\left(\begin{array}{cccc} a_{1,1} & a_{1,2} & a_{1,4} & a_{1,6} \\ a_{1,2} & a_{2,2} & a_{2,4} & a_{2,6} \\ a_{1,4} & a_{2,4} & a_{4,4} & a_{4,6} \\ a_{1,6} & a_{2,6} & a_{4,6} & a_{6,6} \end{array}\right)$,
  \end{center}
  and note if we require it be singular, that determines a unique $a_{2,2}$, with $deg(a_{2,2}) = 0$. This means that every row of $\tilde{A}'$ can be written as a linear combination of rows $1$, $4$, and $6$, and therefore $\tilde{A}'$ has a rank three lift.
  
  As $A'$ has a symmetric rank three lift, so does $A$, and our proof is complete.
\end{proof}

\subsection{Searching for joints}

Note that in this subsection we will, throughout, assume that $A$ is a symmetric matrix with symmetric tropical rank at most three.

The proof that, with one exception, if every $4 \times 4$ submatrix of a $5 \times 5$ symmetric matrix is symmetrically tropically singular then the matrix must have joints involves the analysis of a number of cases, and will be broken down into many lemmas.

Before we go through these cases and prove these lemmas, we will need an additional fact concerning the permutations that realize the tropical determinant of a symmetrically singular $5 \times 5$ matrix.

\begin{lemma}
  If $A$ is a $5 \times 5$ symmetrically tropically singular matrix, then there is a permutation with a tranposition in its disjoint cycle decomposition realizing the tropical determinant.
\end{lemma}

\begin{proof}
  If $\sigma$ realizes the tropical determinant of $A$, then if $\sigma$ has a $2$-cycle in its cycle decomposition there is nothing to prove. If the cycle decomposition of $\sigma$ has a $4$-cycle then by Lemma 1 there must also be a permutation realizing the tropical determinant that is the product of two transpositions. As for the other possibilities, after perhaps a diagonal permutation, the matrix $A$ must have one of the following forms:
  \begin{center}
    
    \begin{description}
      
    \item[Identity] : $\left(\begin{array}{ccccc} 0 & & & & \\ & 0 & & & \\ & & 0 & & \\ & & & 0 & \\ & & & & 0 \end{array}\right)$,
      
    \item[$3$-cycle] : $\left(\begin{array}{ccccc} 0 & & & & \\ & 0 & & & \\ & & & 0 & 0 \\ & & 0 & & 0 \\ & & 0 & 0 & \end{array}\right)$,
      
    \item[$5$-cycle] : $\left(\begin{array}{ccccc} & 0 & & & 0 \\ 0 & & 0 & & \\ & 0 & & 0 & \\ & & 0 & & 0 \\ 0 & & & 0 & \end{array}\right)$.
      
    \end{description}
    
  \end{center}    
  If $A$ is symmetrically singular then each of these matrices must have an additional $0$ term that is not specified above, and for any of these possibilities an additional $0$ term will introduce a permutation realizing the tropical determinant with a cycle decomposition that includes a transposition.
\end{proof}

After possibly a diagonal permutation, we may assume the upper-left $2 \times 2$ submatrix of $A$ has the form:
\begin{center}
  $\left(\begin{array}{cc} & 0 \\ 0 & \end{array}\right)$,
\end{center}
and there is a permutation that realizes the tropical determinant of $A$ whose disjoint cycle decomposition includes the transposition $(12)$.

\begin{lemma}
  If $A$ has symmetric tropical rank three, and does not have a permutation realizing the tropical determinant whose disjoint cycle decomposition is a product of transpositions, then $A$ has joints.
\end{lemma}

\begin{proof}
  As $A$ must have a permutation realizing the determinant that involves the transposition $(12)$, the only possibilities for this minimizing permutation are $(12)$ and $(12)(345)$, which would give $A$ the form:
  \begin{center}
    $\left(\begin{array}{ccccc}  & 0 & & & \\ 0 & & & & \\ & & 0 & + & + \\ & & + & 0 & + \\ & & + & + & 0 \end{array}\right)$, \hspace{.1 in} or \hspace{.1 in} $\left(\begin{array}{ccccc} & 0 & & & \\ 0 & & & & \\ & & + & 0 & 0 \\ & & 0 & + & 0 \\ & & 0 & 0 & + \end{array}\right)$.
  \end{center}
  
  In the first possibility the submatrix $A_{12}$ has the form:
  \begin{center}
    $\left(\begin{array}{cccc} 0 & & & \\ & 0 & + & + \\ & + & 0 & + \\ & + & + & 0 \end{array}\right)$.
  \end{center}
    
  The submatrix $A_{12}$ must be singular, and so there must be another $0$ term in the first row, and a corresponding $0$ term in the first column. By corresponding, we mean that if the $0$ in the first row of $A_{12}$ is in the $i$th column, then the $0$ in the first column of $A_{12}$ must be in the $i$th row. Taking this into account, after possibly a diagonal permutation, $A$ will have the form:
  \begin{center}
    $\left(\begin{array}{ccccc} & 0 & 0 & & \\ 0 & & 0 & & \\ 0 & 0 & 0 & + & + \\ & & + & 0 & + \\ & & + & + & 0 \end{array}\right)$.
  \end{center}
  
  The submatrix $A_{11}$ is
  \begin{center}
    $\left(\begin{array}{cccc} & 0 & & \\ 0 & 0 & + & + \\ & + & 0 & + \\ & + & + & 0 \end{array}\right)$.
  \end{center}  
  As this submatrix must be symmetrically tropically singular we can see from its form that there must be two permutations realizing the tropical determinant, one (noting $A_{11}$ inherits its indices from $A$) whose disjoint cycle decomposition contains the transposition $(23)$, and another whose disjoint cycle decomposition does not. The same will be true, mutatis mutandis, of the submatrix $A_{22}$. From this we can see $A$ has joints $1$ and $2$.
  
  As for the second possibility, the submatrix $A_{12}$ will have the form:
  \begin{center}
    $\left(\begin{array}{cccc} 0 & & & \\ & + & 0 & 0 \\ & 0 & + & 0 \\ & 0 & 0 & + \end{array}\right)$.
  \end{center}
  $A_{12}$ must be symmetrically tropically singular, and so there must be an additional $0$ term in the first row, and an additional $0$ term in the first column. Noting this, after possibly a diagonal permutation, the matrix $A$ must have one of the forms:
  \begin{center}
    $\left(\begin{array}{ccccc} & 0 & 0 & & \\ 0 & & 0 & & \\ 0 & 0 & + & 0 & 0 \\ & & 0 & + & 0 \\ & & 0 & 0 & + \end{array}\right)$, \hspace{.1 in} or \hspace{.1 in} $\left(\begin{array}{ccccc} & 0 & 0 & & \\ 0 & & & 0 & \\ 0 & & + & 0 & 0 \\ & 0 & 0 & + & 0 \\ & & 0 & 0 & + \end{array}\right)$.
  \end{center}
  Using essentially identical reasoning as in the first possibility, we find that $A$ will have joints $1$ and $2$.
\end{proof}

We now examine the possibilities for when there is a permutation that realizes the tropical determinant of $A$ with a disjoint cycle decomposition that is the product of two transpositions. After possibly a diagonal permutation, we may assume this disjoint cycle decomposition is $(12)(34)$.

\begin{lemma}
  Suppose the matrix $A$ has the form:
  \begin{center}
    $\left(\begin{array}{ccccc} + & 0 & & & \\ 0 & + & & & \\ & & + & 0 & \\ & & 0 & + & \\ & & & & 0 \end{array}\right)$.
  \end{center}
  Then $A$ has joints.
\end{lemma}

\begin{proof}
  The submatrix $A_{55}$ must be symmetrically tropically singular, and therefore, after possibly a diagonal permutation, it must have the form
  \begin{center}
    $\left(\begin{array}{cccc} + & 0 & & 0 \\ 0 & + & 0 & \\ & 0 & + & 0 \\ 0 & & 0 & + \end{array}\right)$.  
  \end{center}
  After another diagonal permutation $A_{55}$ can be arranged to have the form
  \begin{center}
    $\left(\begin{array}{cccc} & & 0 & 0 \\ & & 0 & 0 \\ 0 & 0 & & \\ 0 & 0 & & \end{array}\right)$.
  \end{center}
  The matrix $A$ will have the corresponding form
  \begin{center}
    $\left(\begin{array}{ccccc} & & 0 & 0 & \\ & & 0 & 0 & \\ 0 & 0 & & & \\ 0 & 0 & & & \\ & & & & 0 \end{array}\right)$.
  \end{center}
  This is a form that will come up as a possibility in other cases, and we will refer to it as \emph{off-diagonal form}. We will complete our lemma by proving that any matrix in off-diagonal form must have joints.  
  
  If $A$ has off-diagonal form, the submatrix $A_{11}$ will have the form:
  \begin{center}
    $\left(\begin{array}{cccc} & 0 & 0 & \\ 0 & & & \\ 0 & & & \\ & & & 0 \end{array}\right)$.
  \end{center}
  This submatrix must be symmetrically tropically singular. Denote by $M$ the minimal element in the $2 \times 2$ submatrix formed by rows $3$ and $4$, and columns $3$ and $4$ (recall $A_{11}$ inherits its indices from $A$), and denote by $N$ the minimal element in the $2 \times 1$ submatrix formed by rows $3$ and $4$, and column $5$. Suppose $M < 2N$. Given $A_{11}$ is symmetrically tropically singular it must, up to a diagonal permutation, have one of the two forms:
  \begin{center}
    $\left(\begin{array}{cccc} & 0 & 0 & \\ 0 & M & & \\ 0 & & M & \\ & & & 0 \end{array}\right)$, \hspace{.1 in} $\left(\begin{array}{cccc} & 0 & 0 & \\ 0 & M & M & \\ 0 & M & & \\ & & & 0 \end{array}\right)$.
  \end{center}
  In either case the submatrix $A_{11}$ satisfies the joint requirement for joints $1$ and $3$.
  
  If $M = 2N$ then, again given $A_{11}$ is symmetrically tropically singular, it must have, up to a diagonal permutation, one of the three forms:  
  \begin{center}
    $\left(\begin{array}{cccc} & 0 & 0 & \\ 0 & 2N & & \\ 0 & & & N \\ & & N & 0 \end{array}\right)$, \hspace{.1 in} $\left(\begin{array}{cccc} & 0 & 0 & \\ 0 & 2N & & N \\ 0 & & & \\ & N & & 0 \end{array}\right)$, \hspace{.1 in} $\left(\begin{array}{cccc} & 0 & 0 & \\ 0 & & 2N & N \\ 0 & 2N & & \\ & N & & 0 \end{array}\right)$.
  \end{center}
  In either case the submatrix $A_{11}$ again satisfies the joint requirement for joints $1$ and $3$.
  
  Finally, if $M > 2N$ then as $A_{11}$ is symmetrically tropically singular it must have the form:    
  \begin{center}
    $\left(\begin{array}{cccc} & 0 & 0 & \\ 0 & & & N \\ 0 & & & N \\ & N & N & 0 \end{array}\right)$.
  \end{center}
  In this case, again, the submatrix $A_{11}$ satisfies the joint requirement for joints $1$ and $3$.
  
  In each of these six possibilities $A_{11}$ satisfies the joint requirement for joints $1$ and $3$. An identical analysis can be performed on the submatrix $A_{33}$, and from this we can get that $A$ has joints $1$ and $3$. So, any matrix with off-diagonal form has joints.
\end{proof}

\begin{lemma}
  Suppose $A$ has the form:
  \begin{center}
    $\left(\begin{array}{ccccc} + & 0 & & & \\ 0 & + & & & \\ & & + & 0 & \\ & & 0 & 0 & \\ & & & & 0 \end{array}\right)$.
  \end{center}
  Then $A$ has joints.
\end{lemma}

\begin{proof}
  The submatrix $A_{55}$ must be symmetrically tropically singular, and this means either there is a diagonal permutation that will put $A$ in off-diagonal form, in which case we are done, or $A_{55}$ has the form:
  \begin{center}
    $\left(\begin{array}{cccc} + & 0 & 0 & + \\ 0 & + & 0 & + \\ 0 & 0 & + & 0 \\ + & + & 0 & 0 \end{array}\right)$.
  \end{center}
  In this case $A_{33}$ must have the form:
  \begin{center}
    $\left(\begin{array}{cccc} + & 0 & + & \\ 0 & + & + & \\ + & + & 0 & \\ & & & 0 \end{array}\right)$.
  \end{center}
  As $A_{33}$ must be symmetrically tropically singular, it must have one of the following two forms:
  \begin{center}
    $\left(\begin{array}{cccc} + & 0 & + & 0 \\ 0 & + & + & 0 \\ + & + & 0 & \\ 0 & 0 & & 0 \end{array}\right)$, \hspace{.1 in} or \hspace{.1 in} $\left(\begin{array}{cccc} + & 0 & + & \\ 0 & + & + & \\ + & + & 0 & 0 \\ & & 0 & 0 \end{array}\right)$.
  \end{center}    
  In the first possibility $A$ has the form:
  \begin{center}
    $\left(\begin{array}{ccccc} + & 0 & 0 & + & 0 \\ 0 & + & 0 & + & 0 \\ 0 & 0 & + & 0 & \\ + & + & 0 & 0 & \\ 0 & 0 & & & 0 \end{array}\right)$.
  \end{center}
  This form has joints $1$ and $2$. In the second possibility $A$ has the form:
  \begin{center}
    $\left(\begin{array}{ccccc} + & 0 & 0 & + &  \\ 0 & + & 0 & + &  \\ 0 & 0 & + & 0 & \\ + & + & 0 & 0 & 0 \\ & & & 0 & 0 \end{array}\right)$.
  \end{center}
  The submatrix $A_{44}$ has the form:
  \begin{center}
    $\left(\begin{array}{cccc} + & 0 & 0 & \\ 0 & + & 0 & \\ 0 & 0 & + & \\ & & & 0 \end{array}\right)$.
  \end{center}
  This submatrix must be symmetrically tropically singular and therefore, up to a diagonal permutation, must have the form:
  \begin{center}
    $\left(\begin{array}{cccc} + & 0 & 0 & 0 \\ 0 & + & 0 & \\ 0 & 0 & + & \\ 0 & & & 0 \end{array}\right)$.    
  \end{center}
  The corresponding form for $A$ is:
  \begin{center}
    $\left(\begin{array}{ccccc} + & 0 & 0 & + & 0 \\ 0 & + & 0 & + &  \\ 0 & 0 & + & 0 & \\ + & + & 0 & 0 & 0 \\ 0 & & & 0 & 0 \end{array}\right)$.
  \end{center}   
  Any matrix of this form has joints $1$ and $2$.
\end{proof}

\begin{lemma}
  Suppose $A$ has the form:
  \begin{center}
    $\left(\begin{array}{ccccc} + & 0 & & & \\ 0 & + & & & \\ & & 0 & 0 & \\ & & 0 & 0 & \\ & & & & 0 \end{array}\right)$.
  \end{center}
  Then $A$ has joints.
\end{lemma}

\begin{proof}
  Suppose $A$ has the form
  \begin{center}
    $\left(\begin{array}{ccccc} + & 0 & & & \\ 0 & + & & & \\ & & 0 & 0 & + \\ & & 0 & 0 & + \\ & & + & + & 0 \end{array}\right)$.
  \end{center}
  The submatrices $A_{33}$ and $A_{44}$ will have the form:
  \begin{center}
    $\left(\begin{array}{cccc} + & 0 & & \\ 0 & + & & \\ & & 0 & + \\ & & + & 0 \end{array}\right)$.
  \end{center}
  For these submatrices to be symmetrically tropically singular they must have, up to a diagonal permutation, one of the two forms:
  \begin{center}
    $\left(\begin{array}{cccc} + & 0 & 0 & \\ 0 & + & 0 & \\ 0 & 0 & 0 & + \\ & & + & 0 \end{array}\right)$, \hspace{.1 in} or \hspace{.1 in} $\left(\begin{array}{cccc} + & 0 & 0 & \\ 0 & + & & 0 \\ 0 & & 0 & + \\ & 0 & + & 0 \end{array}\right)$.
  \end{center}
  Based on these forms, the matrix $A$, possibly after a diagonal permutation, must either have off-diagonal form, in which case we are done, or have one of the following two forms:
  \begin{center}
    $\left(\begin{array}{ccccc} + & 0 & 0 & 0 & \\ 0 & + & & & 0 \\ 0 & & 0 & 0 & \\ 0 & & 0 & 0 & \\ & 0 & & & 0  \end{array}\right)$, \hspace{.1 in} or \hspace{.1 in} $\left(\begin{array}{ccccc} + & 0 & & & 0 \\ 0 & + & & & 0 \\ & & 0 & 0 & \\ & & 0 & 0 & \\ 0 & 0 & & & 0 \end{array}\right)$.
  \end{center}
  The first possibility has joints $3$ and $4$.
    
  Suppose $A$ has the form of our second possibility above. Denote by $M$ the minimal off-diagonal term in $A$ that is not necessarily $0$. If $M$ is in the $2 \times 2$ submatrix formed by rows $1$ and $2$, and columns $3$ and $4$ then, after possibly a diagonal permutation, $A$ will have the form:
  \begin{center}
    $\left(\begin{array}{ccccc} + & 0 & & & 0 \\ 0 & + & M & & 0 \\ & M & 0 & 0 & \\ & & 0 & 0 & \\ 0 & 0 & & & 0 \end{array}\right)$.
  \end{center}
  Given the submatrix $A_{32}$ must be symmetrically tropically singular, $A$ must have one of the following five forms:    
  \begin{center}
    $\left(\begin{array}{ccccc} + & 0 & M & & 0 \\ 0 & + & M & & 0 \\ M & M & 0 & 0 & \\ & & 0 & 0 & \\ 0 & 0 & & & 0 \end{array}\right)$, \hspace{.1 in} $\left(\begin{array}{ccccc} + & 0 & & M & 0 \\ 0 & + & M & & 0 \\ & M & 0 & 0 & \\ M & & 0 & 0 & \\ 0 & 0 & & & 0 \end{array}\right)$, 
    
    $\left(\begin{array}{ccccc} + & 0 & & & 0 \\ 0 & + & M & M & 0 \\ & M & 0 & 0 & \\ & M & 0 & 0 & \\ 0 & 0 & & & 0 \end{array}\right)$, \hspace{.1 in} $\left(\begin{array}{ccccc} + & 0 & & & 0 \\ 0 & + & M & & 0 \\ & M & 0 & 0 & M \\ & & 0 & 0 & \\ 0 & 0 & M & & 0 \end{array}\right)$, 
    
    $\left(\begin{array}{ccccc} + & 0 & & & 0 \\ 0 & + & M & & 0 \\ & M & 0 & 0 & \\ & & 0 & 0 & M \\ 0 & 0 & & M & 0 \end{array}\right)$. 
  \end{center}
  All these possibilities have joints $2$ and $3$.
  
  If $M$ is not in that $2 \times 2$ submatrix, then, possibly after a diagonal permutation, we may assume $a_{4,5} = M$. As $A_{45}$ must be symmetrically tropically singular we get that $A$ must have the form:
  \begin{center}
    $\left(\begin{array}{ccccc} + & 0 & & & 0 \\ 0 & + & & & 0 \\ & & 0 & 0 & M \\ & & 0 & 0 & M \\ 0 & 0 & M & M & 0 \end{array}\right)$.
  \end{center}
  This form has joints $4$ and $5$.
  
  If $A$ has the form:
  \begin{center}
    $\left(\begin{array}{ccccc} + & 0 & & & \\ 0 & + & & & \\ & & 0 & 0 & + \\ & & 0 & 0 & 0 \\ & & + & 0 & 0 \end{array}\right)$,
  \end{center}
  then, given the submatrix $A_{44}$ must be symmetrically tropically singular, the possible forms of $A$, up to diagonal permutation, that are distinct from ones we have already examined are:
  \begin{center}
    $\left(\begin{array}{ccccc} + & 0 & 0 & & \\ 0 & + & & & 0 \\ 0 & & 0 & 0 & + \\ & & 0 & 0 & 0 \\ & 0 & + & 0 & 0 \end{array}\right)$, \hspace{.1 in } or \hspace{.1 in} $\left(\begin{array}{ccccc} + & 0 & 0 & & \\ 0 & + & 0 & & \\ 0 & 0 & 0 & 0 & + \\ & & 0 & 0 & 0 \\ & & + & 0 & 0 \end{array}\right)$.
  \end{center}
  The first possibility has joints $3$ and $4$. In the second possibility we note that the submatrix  $A_{51}$ is    
  \begin{center}
    $\left(\begin{array}{cccc} 0 & 0 & & \\ + & 0 & & \\ 0 & 0 & 0 & + \\ & 0 & 0 & 0 \end{array}\right)$.
  \end{center}  
  This matrix must be symmetrically tropically singular, and therefore, up to diagonal permutation, the matrix $A$ must have one of the forms: 
  \begin{center}
    $\left(\begin{array}{ccccc} + & 0 & 0 & & \\ 0 & + & 0 & 0 & \\ 0 & 0 & 0 & 0 & + \\ & 0 & 0 & 0 & 0 \\ & & + & 0 & 0 \end{array}\right)$, \hspace{.1 in} or \hspace{.1 in} $\left(\begin{array}{ccccc} + & 0 & 0 & & \\ 0 & + & 0 & & 0 \\ 0 & 0 & 0 & 0 & + \\ & & 0 & 0 & 0 \\ & 0 & + & 0 & 0 \end{array}\right)$.
  \end{center}
  Both have joints $3$ and $4$.
  
  Finally, suppose $A$ has the form  
  \begin{center}
    $\left(\begin{array}{ccccc} + & 0 & & & \\ 0 & + & & & \\ & & 0 & 0 & 0 \\ & & 0 & 0 & 0 \\ & & 0 & 0 & 0 \end{array}\right)$.
  \end{center}
  This matrix has the form of the first example matrix we examined, and so has joints $4$ and $5$. This exhausts all the possible forms of $A$, given the requirements of the lemma, and we have demonstrated that all these possibilities have joints.
\end{proof}

\begin{lemma}
  Suppose $A$ has a permutation realizing the tropical determinant whose disjoint cycle decomposition is the product of two transpositions, and after a diagonal permutation it can be arranged so this permutation realizing the tropical determinant has cycle-decomposition $(12)(34)$, and the upper-left $2 \times 2$ submatrix of $A$ has the form:
  \begin{center}
    $\left(\begin{array}{cc} + & 0 \\ 0 & + \end{array}\right)$.
  \end{center}
  Then $A$ has joints.
\end{lemma}

\begin{proof}
  All the possible forms of $A$ that satisfy these requirements are handled by Lemmas 4, 5, and 6. Therefore, $A$ has joints.
\end{proof}

\begin{lemma}
  Suppose $A$ has a permutation realizing the tropical determinant whose disjoint cycle decomposition is the product of two transpositions, it is possible to find a diagonal permutation such that the permutation realizing the tropical determinant is $(12)(34)$, and the upper-left $2 \times 2$ submatrix is:
  \begin{center}
    $\left(\begin{array}{cc} + & 0 \\ 0 & 0 \end{array}\right)$,
  \end{center} 
  while it is impossible to find a diagonal permutation such that the permutation realizing the tropical determinant is $(12)(34)$, and the upper-left $2 \times 2$ submatrix is:  
  \begin{center}
    $\left(\begin{array}{cc} + & 0 \\ 0 & + \end{array}\right)$.
  \end{center}
  Then $A$ has joints.
\end{lemma}

\begin{proof}
  If $A$ has the form:  
  \begin{center}
    $\left(\begin{array}{ccccc} + & 0 & & & \\ 0 & 0 & & & \\ & & + & 0 & \\ & & 0 & 0 & \\ & & & & 0 \end{array}\right)$,
  \end{center}
  then as $A_{55}$ must be symmetrically tropically singular the only possibility is that $A$ has off-diagonal form. 
  
  Suppose $A$ has the form:
  \begin{center}
      $\left(\begin{array}{ccccc} + & 0 & & & 0 \\ 0 & 0 & & & \\ & & 0 & 0 & \\ & & 0 & 0 & \\ 0 & & & & 0\end{array}\right)$.
  \end{center}  
  Let $M$ denote the minimal element that is not necessarily $0$ and is not $a_{2,5}$ or its symmetric counterpart $a_{5,2}$. If $a_{i,j} = M$ then, given the submatrix $A_{ij}$ must be symmetrically tropically singular, we can derive that, up to a diagonal permutation, $A$ must have one of the following nine forms: 
  \begin{center}
    $\left(\begin{array}{ccccc} + & 0 & M & & 0 \\ 0 & 0 & M & & \\ M & M & 0 & 0 & \\ & & 0 & 0 & \\ 0 & & & & 0 \end{array}\right)$, \hspace{.1 in} $\left(\begin{array}{ccccc} + & 0 & & M & 0 \\ 0 & 0 & M & & \\  & M & 0 & 0 & \\ M & & 0 & 0 & \\ 0 & & & & 0 \end{array}\right)$,
    
    $\left(\begin{array}{ccccc} + & 0 & & & 0 \\ 0 & 0 & M & M & \\ & M & 0 & 0 & \\ & M & 0 & 0 & \\ 0 & & & & 0 \end{array}\right)$, \hspace{.1 in} $\left(\begin{array}{ccccc} + & 0 & & & 0 \\ 0 & 0 & M & & \\ & M & 0 & 0 & M \\ & & 0 & 0 & \\ 0 & & M & & 0 \end{array}\right)$,
    
    $\left(\begin{array}{ccccc} + & 0 & & & 0 \\ 0 & 0 & M & & \\ & M & 0 & 0 & \\ & & 0 & 0 & M \\ 0 & & & M & 0 \end{array}\right)$, \hspace{.1 in} $\left(\begin{array}{ccccc} + & 0 & M & M & 0 \\ 0 & 0 & & & \\ M & & 0 & 0 & \\ M & & 0 & 0 & \\ 0 & & & & 0 \end{array}\right)$,
    
    $\left(\begin{array}{ccccc} + & 0 & M & & 0 \\ 0 & 0 & & & \\ M & & 0 & 0 & M \\ & & 0 & 0 & \\ 0 & & M & & 0 \end{array}\right)$, \hspace{.1 in} $\left(\begin{array}{ccccc} + & 0 & M & & 0 \\ 0 & 0 & & & \\ M & & 0 & 0 & \\ & & 0 & 0 & M \\ 0 & & & M & 0 \end{array}\right)$,
        
    $\left(\begin{array}{ccccc} + & 0 & & & 0 \\ 0 & 0 & & & \\ & & 0 & 0 & M \\ & & 0 & 0 & M \\ 0 & & M &M & 0 \end{array}\right)$.  
  \end{center}
  The first five possibilities have joints $2$ and $3$, possibilities six through eight have joints $1$ and $3$, and the ninth possibility has joints $4$ and $5$. 
    
  Suppose $A$ has the form:  
  \begin{center}
    $\left(\begin{array}{ccccc} + & 0 & & & + \\ 0 & 0 & & & \\ & & 0 & 0 & + \\ & & 0 & 0 & + \\ + & & + & + & 0\end{array}\right)$.
  \end{center}
  Given the submatrices $A_{33}$ and $A_{44}$ must be symmetrically tropically singular, the matrix $A$ must have the form:
  \begin{center}
    $\left(\begin{array}{ccccc} + & 0 & 0 & 0 & + \\ 0 & 0 & & & \\ 0 & & 0 & 0 & + \\ 0 & & 0 & 0 & + \\ + & & + & + & 0 \end{array}\right)$.
  \end{center}
  This matrix has joints $3$ and $4$. 
  
  If $A$ has the form:  
  \begin{center}
    $\left(\begin{array}{ccccc} + & 0 & & & + \\ 0 & 0 & & & \\ & & 0 & 0 & + \\ & & 0 & 0 & 0 \\ + & & + & 0 & 0\end{array}\right)$,
  \end{center}
  then, given $A_{44}$ must be symmetrically tropically singular, $A$ must have the form:
  \begin{center}
    $\left(\begin{array}{ccccc} + & 0 & 0 & & + \\ 0 & 0 & & & \\ 0 & & 0 & 0 & + \\ & & 0 & 0 & 0 \\ + & & + & 0 & 0\end{array}\right)$,
  \end{center}
  which, after a diagonal permutation, is a form analyzed earlier in the lemma.
    
  Finally, if $A$ has the form:  
  \begin{center}
    $\left(\begin{array}{ccccc} + & 0 & & & + \\ 0 & 0 & & & \\ & & 0 & 0 & 0 \\ & & 0 & 0 & 0 \\ + & & 0 & 0 & 0\end{array}\right)$,
  \end{center}
  then it is of the form of the example matrix we first analyzed in this section, and $A$ has joints $4$ and $5$. This exhausts all the possibilities, and the lemma is proven.
\end{proof}

Up to diagonal permutation the only form we have yet to consider is:
\begin{center}
  $\left(\begin{array}{ccccc} 0 & 0 & & & \\ 0 & 0 & & & \\ & & 0 & 0 & \\ & & 0 & 0 & \\ & & & & 0 \end{array}\right)$.
\end{center}
Denote the minimal term in the $2 \times 2$ submatrix formed by rows $1, 2$ and columns $3,4$ as $M$, the minimal term in the $2 \times 1$ submatrix formed by rows $1,2$ and column $5$ as $N$, and the minimal term in the $2 \times 1$ submatrix formed by rows $3,4$ and column $5$ as $P$.
  
If either $N = 0$ or $P = 0$ we can, after a diagonal permutation, assume $A$ has the form:

\begin{center}    
  $\left(\begin{array}{ccccc} 0 & 0 & & & \\ 0 & 0 & & & \\ & & 0 & 0 & \\ & & 0 & 0 & 0 \\ & & & 0 & 0 \end{array}\right)$.
\end{center}

\begin{lemma}
  If $A$ has the form:    
  \begin{center}
    $\left(\begin{array}{ccccc} 0 & 0 & & & \\ 0 & 0 & & & \\ & & 0 & 0 & \\ & & 0 & 0 & 0 \\ & & & 0 & 0 \end{array}\right)$,
  \end{center}
  then $A$ has joints.
\end{lemma}

\begin{proof}
  If $A$ has the form:    
  \begin{center}
    $\left(\begin{array}{ccccc} 0 & 0 & & & \\ 0 & 0 & & & \\ & & 0 & 0 & 0 \\ & & 0 & 0 & 0 \\ & & 0 & 0 & 0 \end{array}\right)$,
  \end{center}  
  then $A$ has the form of the first example analyzed in this chapter, and therefore has joints $4$ and $5$. 
  
  Suppose $A$ has the form:  
  \begin{center}
    $\left(\begin{array}{ccccc} 0 & 0 & & & \\ 0 & 0 & & & \\ & & 0 & 0 & + \\ & & 0 & 0 & 0 \\ & & + & 0 & 0 \end{array}\right)$,
  \end{center}
  and denote by $M$ the minimal term in $A$ that is not necessarily $0$ and is not the term $a_{3,5}$ or its symmetric counterpart $a_{5,3}$. Then, given the submatrix $A_{42}$ has the form:
  \begin{center}
    $\left(\begin{array}{cccc} 0 & & & \\ 0 & & & \\ & 0 & 0 & + \\ & + & 0 & 0 \end{array}\right)$,
  \end{center}
  and must be symmetrically tropically singular, $A$ must have, up to diagonal permutation, one of the following six forms:
  \begin{center}
    $\left(\begin{array}{ccccc} 0 & 0 & & M & \\ 0 & 0 & M & & \\ & M & 0 & 0 & + \\ M & & 0 & 0 & 0 \\ & & + & 0 & 0 \end{array}\right)$, \hspace{.1 in} $\left(\begin{array}{ccccc} 0 & 0 & M & & \\ 0 & 0 & M & & \\ M & M & 0 & 0 & + \\ & & 0 & 0 & 0 \\ & & + & 0 & 0 \end{array}\right)$,
      
    $\left(\begin{array}{ccccc} 0 & 0 & & & \\ 0 & 0 & M & M & \\ & M & 0 & 0 & + \\ & M & 0 & 0 & 0 \\ & & + & 0 & 0 \end{array}\right)$, \hspace{.1 in} $\left(\begin{array}{ccccc} 0 & 0 & & & M \\ 0 & 0 & M & & \\ & M & 0 & 0 & + \\ & & 0 & 0 & 0 \\ M & & + & 0 & 0 \end{array}\right)$,
        
    $\left(\begin{array}{ccccc} 0 & 0 & & & \\ 0 & 0 & M & & M \\ & M & 0 & 0 & + \\ & & 0 & 0 & 0 \\ & M & + & 0 & 0 \end{array}\right)$, \hspace{.1 in} $\left(\begin{array}{ccccc} 0 & 0 & & & M \\ 0 & 0 & & & M \\ & & 0 & 0 & + \\ & & 0 & 0 & 0 \\ M & M & + & 0 & 0 \end{array}\right)$.   
  \end{center}
  The first five possibilities have joints $2$ and $3$. The final possibility has joints $1$ and $5$.
\end{proof}
    
\begin{lemma}
  Suppose $A$ has the form 
  \begin{center}
    $\left(\begin{array}{ccccc} 0 & 0 & & & \\ 0 & 0 & & & \\ & & 0 & 0 & \\ & & 0 & 0 & \\ & & & & 0 \end{array}\right)$.
  \end{center}
  Denote the minimal term in the $2 \times 2$ submatrix formed by rows $1, 2$ and columns $3,4$ as $M$, the minimal term in the $2 \times 1$ submatrix formed by rows $1,2$ and column $5$ as $N$, and the minimal term in the $2 \times 1$ submatrix formed by rows $3,4$ and column $5$ as $P$. If $M \leq N \otimes P$ then $A$ has joints.
\end{lemma}

\begin{proof}
  After possibly a diagonal permutation we may assume $A$ has the form:  
  \begin{center}
    $\left(\begin{array}{ccccc} 0 & 0 & & & \\ 0 & 0 & M & & \\ & M & 0 & 0 & \\ & & 0 & 0 & \\ & & & & 0 \end{array}\right)$.
  \end{center}
  The submatrix $A_{32}$ has the form
  \begin{center}
    $\left(\begin{array}{cccc} 0 & & & \\ 0 & M & & \\ & 0 & 0 & \\ & & & 0 \end{array}\right)$
  \end{center}
  and must be symmetrically tropically singular. Therefore, $A_{32}$ must have two distinct permutations realizing the tropical determinant, one involving the $M$ term and the other not. Therefore $A$ has joints $2$ and $3$.
  \end{proof}
    
  Finally, if $A$ has the form
  \begin{center}
    $\left(\begin{array}{ccccc} 0 & 0 & & & \\ 0 & 0 & & & \\ & & 0 & 0 & \\ & & 0 & 0 & \\ & & & & 0 \end{array}\right)$,
  \end{center}
  with $M$,$N$, and $P$ defined as in the lemma above, if $N,P > 0$ and $N \otimes P < M$, then possibly after a diagonal permutation we may assume $A$ has the form
  \begin{center}
    $\left(\begin{array}{ccccc} 0 & 0 & & & N \\ 0 & 0 & & & N' \\ & & 0 & 0 & P \\ & & 0 & 0 & P' \\ N & N' & P & P' & 0 \end{array}\right)$,
  \end{center}
  with $N \leq N'$ and $P \leq P'$. If we examine the submatrix $A_{31}$ we get
  \begin{center}
    $\left(\begin{array}{cccc} 0 & & & N \\ 0 & & & N' \\ & 0 & 0 & P' \\ N' & P & P' & 0 \end{array}\right)$.
  \end{center}
  In order for this sumbatrix to be symmetrically tropically singular, we must have either $N = N'$ or $P = P'$. Possibly after a diagonal permutation, we may assume $P = P'$. This is the exceptional form examined earlier in this section.

  We've now exhausted the possible forms of $A$, and our results can be summarized in the following proposition.

\begin{proposition}
  If $A$ is a $5 \times 5$ symmetric matrix with symmetric tropical rank three, and if $A$ does not have exceptional form, then $A$ has joints.
\end{proposition}

\begin{proof}
  All the possible forms for $A$, up to diagonal permutation, are proven to have joints by Lemmas 3, 7, 8, 9, and 10.
\end{proof}

We now have all we need to prove the major theorem of this paper.

\begin{thm}
  The $4 \times 4$ minors of a $5 \times 5$ symmetric matrix form a tropical basis.
\end{thm}

\begin{proof}
  This is an immediate consequence of Propositions 2, 3, and 4.
\end{proof}

\section{Larger symmetric matrices}

In \cite{dss}, Develin, Santos, and Sturmfels proved the cocircuit matrix of the Fano matroid,
\begin{center}
  $\left(\begin{array}{ccccccc} 1 & 1 & 0 & 1 & 0 & 0 & 0 \\ 0 & 1 & 1 & 0 & 1 & 0 & 0 \\ 0 & 0 & 1 & 1 & 0 & 1 & 0 \\ 0 & 0 & 0 & 1 & 1 & 0 & 1 \\ 1 & 0 & 0 & 0 & 1 & 1 & 0 \\ 0 & 1 & 0 & 0 & 0 & 1 & 1 \\ 1 & 0 & 1 & 0 & 0 & 0 & 1 \end{array}\right)$,
\end{center}
has tropical rank three but Kapranov rank four. If we permute the rows of this matrix with the permutation given by the disjoint cycle decomposition $(27)(36)(45)$ we get the symmetric matrix
\begin{center}  
  $\left(\begin{array}{ccccccc} 1 & 1 & 0 & 1 & 0 & 0 & 0 \\ 1 & 0 & 1 & 0 & 0 & 0 & 1 \\ 0 & 1 & 0 & 0 & 0 & 1 & 1 \\ 1 & 0 & 0 & 0 & 1 & 1 & 0 \\ 0 & 0 & 0 & 1 & 1 & 0 & 1 \\ 0 & 0 & 1 & 1 & 0 & 1 & 0 \\ 0 & 1 & 1 & 0 & 1 & 0 & 0 \end{array}\right)$.
\end{center}
While this symmetric matrix has standard tropical rank three, its symmetric tropical rank is four, and it's therefore \emph{not} an example of a matrix with symmetric tropical rank three but greater symmetric Kapranov rank.

This matrix can, however, be used to construct the following symmetric matrix with symmetric tropical rank three, but greater symmetric Kapranov rank:
\begin{center}
  $\left(\begin{array}{ccccccccccccc} 0 & 0 & 0 & 0 & 0 & 0 & 1 & 1 & 0 & 1 & 0 & 0 & 0 \\ 0 & 0 & 0 & 0 & 0 & 0 & 1 & 0 & 1 & 0 & 0 & 0 & 1 \\ 0 & 0 & 0 & 0 & 0 & 0 & 0 & 1 & 0 & 0 & 0 & 1 & 1 \\ 0 & 0 & 0 & 0 & 0 & 0 & 1 & 0 & 0 & 0 & 1 & 1 & 0 \\ 0 & 0 & 0 & 0 & 0 & 0 & 0 & 0 & 0 & 1 & 1 & 0 & 1 \\ 0 & 0 & 0 & 0 & 0 & 0 & 0 & 0 & 1 & 1 & 0 & 1 & 0 \\ 1 & 1 & 0 & 1 & 0 & 0 & 0 & 1 & 1 & 0 & 1 & 0 & 0 \\ 1 & 0 & 1 & 0 & 0 & 0 & 1 & 0 & 0 & 0 & 0 & 0 & 0 \\ 0 & 1 & 0 &0 & 0 & 1 & 1 & 0 & 0 & 0 & 0 & 0 & 0 \\ 1 & 0 & 0 & 0 & 1 & 1 & 0 & 0 & 0 & 0 & 0 & 0 & 0 \\ 0 & 0 & 0 & 1 & 1 & 0 & 1 & 0 & 0 & 0 & 0 & 0 & 0 \\ 0 & 0 & 1 & 1 & 0 & 1 & 0 & 0 & 0 & 0 & 0 & 0 & 0 \\ 0 & 1 & 1 & 0 & 1 & 0 & 0 & 0 & 0 & 0 & 0 & 0 & 0 \end{array}\right)$
\end{center}
The upper-right, and bottom-left, $7 \times 7$ submatrices of this $13 \times 13$ symmetric matrix are the symmetric version of the cocircuit matrix of the Fano matroid. This $13 \times 13$ matrix has symmetric tropical rank three. If it had symmetric Kapranov rank three then its upper-right $7 \times 7$ submatrix would have standard Kapranov rank three, and this is impossible.

In \cite{z}, the author proved that if the $r \times r$ minors of an $n \times n$ symmetric matrix of indeterminates are not a tropical basis, then the $r \times r$ minors of an $(n+1) \times (n+1)$ symmetric matrix of indeterminates are also not a tropical basis. From this, we see the $4 \times 4$ minors of an $n \times n$ symmetric matrix of indeterminates are not a tropical basis in $n \geq 13$. The cases $6 \leq n \leq 12$, on the other hand, remain open.

\newtheorem{question}{Question}
\begin{question}
  Are the $4 \times 4$ minors of an $n \times n$ symmetric matrix of indeterminates a tropical basis for $6 \leq n \leq 12$?
\end{question}

While a modified version of the approach used in this paper could work for larger symmetric matrices, the number of cases that would need to be checked is significant, and checking them all would likely be arduous. In Section 3 of \cite{cjr} Chan, Jensen, and Rubei compute the set of $5 \times 5$ matrices of tropical rank at most 3, and of Kapranov rank at most 3, using the software Gfan. They then compare the sets and show they are equal. The most straightforward way to answer Question 1 might be to use Gfan or some other computational software package.

\end{document}